\documentclass[a4paper]{amsart}
\usepackage{amsmath}
\usepackage{amssymb}
\usepackage{amsfonts}
\usepackage{pxfonts}
\usepackage[OT2,T1]{fontenc}

\usepackage{pict2e}

\usepackage[
textwidth=14.5cm, 
textheight=21cm,
hmarginratio=1:1,
vmarginratio=1:1]{geometry}

\newtheorem{thm}{Theorem}[section]

\theoremstyle{definition}

\theoremstyle{remark}
\newtheorem{rem}[thm]{Remark}

\numberwithin{equation}{section}
\numberwithin{figure}{section}

\newcommand{\diff}{\mathrm{d}}
\newcommand{\C}{{\mathbb C}}
\newcommand{\R}{{\mathbb R}}
\newcommand{\D}{{\mathbb D}}

\newcommand{\imag}{\mathrm{i}}
\newcommand{\e}{\mathrm{e}}
\newcommand{\hDelta}{\varDelta}
\newcommand{\dbar}{\bar\partial}

\newcommand{\Mop}{{\mathbf M}}

\newcommand{\Tope}{{\mathbf T}}

\DeclareMathOperator{\im}{Im}

\begin{document}

%
\title{On H\"ormander's solution of the $\bar\partial$-equation}


%

\author{Haakan Hedenmalm}
\address{Hedenmalm: Department of Mathematics\\
KTH Royal Institute of Technology\\
S--10044 Stockholm\\
Sweden}

\email{haakanh@math.kth.se}

\subjclass[2000]{Primary }
 
\thanks{The research of the author was supported by the G\"oran Gustafsson 
Foundation (KVA) and by Vetenskapsr\aa{}det (VR)}

\begin{abstract} 
We explain how H\"ormander's classical solution of the $\bar\partial$-equation
in the plane with a weight which permits growth near infinity carries over
to the rather opposite situation when we ask for decay near infinity. Here,
however, a natural condition on the datum needs to be imposed. The condition
is not only natural but also necessary to have the result at least in the
Fock weight case. The norm identity which leads to the estimate is related to
general area-type results in the theory of conformal mappings.  
\end{abstract}

\maketitle

\centerline{\em In memory of Lars H\"ormander}

\section{Introduction} 

\subsection{Basic notation} Let 
\[
\hDelta:=\frac14\bigg(\frac{\partial^2}{\partial x^2}+
\frac{\partial^2}{\partial y^2}\bigg),\qquad\diff A(z):=\diff x\diff y,
\]
denote the normalized Laplacian and the area element, respectively. 
Here, $z=x+\imag y$
is the standard decomposition into real and imaginary parts. We let $\C$ denote
the complex plane. We also need the standard complex differential operators
\[
\dbar_z:=\frac{1}{2}\bigg(\frac{\partial}{\partial x}+\imag
\frac{\partial}{\partial y}\bigg),\quad \partial_z:=
\frac{1}{2}\bigg(\frac{\partial}{\partial x}-\imag
\frac{\partial}{\partial y}\bigg),
\]
so that $\hDelta$ factors as $\hDelta=\partial_z\dbar_z$. We sometimes drop
indication of the differentiation variable $z$.


\subsection{H\"ormander's solution of the $\bar\partial$-problem} 
We present H\"ormander's theorem \cite{Horm1}, \cite{Horm2} in the simplest 
possible case, when the domain is the entire complex plane and the weight 
$\phi:\C\to\R$ is $C^{2}$-smooth with $\hDelta\phi>0$ everywhere.

\begin{thm}
{\rm(H\"ormander)} If the complex-valued function $f$ is locally area 
$L^2$-integrable in the plane $\C$, then there exists a solution to the 
$\dbar$-equation $\dbar u=f$ with
\[
\int_\C|u|^2\e^{-2\phi}\diff A\le\frac12
\int_\C|f|^2\frac{\e^{-2\phi}}{\hDelta\phi}\diff A.
\]
\label{thm-H1}
\end{thm}

Here, we remark that the assertion of the theorem is void unless the integral
on the right hand side is finite.

\subsection{The $\dbar$-equation with growing weights}
While Theorem \ref{thm-H1} essentially deals with decaying weights 
$\e^{-2\phi}$, it is natural to ask what happens if we were to consider the
growing weights $\e^{2\phi}$ in place of $\e^{-2\phi}$. 
So when could we say that there exists a solution to $\dbar u=f$ with
\begin{equation}
\int_\C|u|^2\e^{2\phi}\diff A\lesssim
\int_\C|f|^2\e^{2\phi}\diff A,
\label{eq-adjest1}
\end{equation}
where the symbol ``$\lesssim$'' is understood liberally? Here, we should have 
the (Fock weight) example $\phi(z)=\frac12|z|^2$ in mind. 
It is rather clear that we 
cannot hope to have an estimate of the type \eqref{eq-adjest1} without an
additional condition on on the datum $f$. For instance, if 
$\phi(z)=\frac12|z|^2$ and $f(z)=\e^{-|z|^2}$, it is not possible to find 
such a fast-decaying function $u$ with $\dbar u=f$ (cf. Section 
\ref{sec-constr} below). 
It is natural to look for a class of data $f$ that would
come from functions $u$ with compact support. Let $C^\infty_c(\C)$ denote the
standard space of infinitely differentiable compactly supported test functions.
The calculation 
\[
\int_\C fg\diff A=\int_\C g\dbar u\diff A=-\int_\C u\dbar g\diff A, 
\qquad u\in C^\infty_c(\C),
\]
shows that the datum $f=\dbar u$ with $u\in C^\infty_c(\C)$ must satisfy,
for entire functions $g$,
\begin{equation}
\int_\C fg\diff A=0.
\label{eq-dual1}
\end{equation}

We remark here that the calculation \eqref{eq-dual1} is the basis for what is
known as Havin's lemma \cite{Hav} (see also, e.g., \cite{Hed1}). 

We let $L^2(\C,\e^{-2\phi})$ and $L^2(\C,\e^{2\phi})$ be the weighted area 
$L^2$-spaces with the indicated weights. The corresponding norms are
\[
\|g\|^2_{L^2(\C,\e^{-2\phi})}=\int_\C|g|^2\e^{-2\phi}\diff A,\quad
\|g\|^2_{L^2(\C,\e^{2\phi})}=\int_\C|g|^2\e^{2\phi}\diff A.
\]
The space of entire functions in $L^2(\C,\e^{-2\phi})$ is denoted by 
$A^2(\C,\e^{-2\phi})$. We recall the standing assumption that $\phi$ be
$C^2$-smooth with $\Delta\phi>0$ everywhere.

\begin{thm}
Suppose $f\in L^2(\C,\e^{2\phi})$ meets the condition \eqref{eq-dual1} for all
$g\in A^2(\C,\e^{-2\phi})$. Then there exists a solution to the 
$\dbar$-equation $\dbar u=f$ with
\[
\int_\C|u|^2\e^{2\phi}\hDelta\phi\diff A\le\frac12
\int_\C|f|^2\e^{2\phi}\diff A.
\]
\label{thm-H2}
\end{thm}

Superficially, this theorem looks quite different than H\"ormander's
Theorem \ref{thm-H1}. However, it is in a sense which can be made precise
dual to Theorem \ref{thm-H1}.
In our presentation, we will derive it from the same rather elementary 
calculation which H\"ormander uses in e.g. \cite{Horm2}, p. 250. 
As such, Theorem \ref{thm-H2} may well be known, but we have not found a 
specific reference. 

We might add that if the polynomials are dense in the space 
$A^2(\C,\e^{-2\phi})$ of entire functions, as they are, e.g., for radial 
$\phi$, 
then the condition \eqref{eq-dual1} needs to be verified only for monomials
$g(z)=z^j$, $j=0,1,2,\ldots$. 

\begin{rem}
Naturally, Theorem \ref{thm-H2} should generalize to the setting of several 
complex variables. 
\end{rem}

Under a simple condition on the weight, the solution supplied by Theorem 
\ref{thm-H2} is unique.
 
\begin{thm}
If, in addition, $\phi$ is $C^4$-smooth and meets the curvature-type condition
\[
\frac{1}{\Delta\phi}\Delta\log\Delta\phi\ge-2 \,\,\,\,\text{on}\,\,\,\C,
\]
then the solution $u$ in Theorem \ref{thm-H2} is unique.
\label{thm-H2.1}
\end{thm}

\section{The proof of Theorem \ref{thm-H2}}

\subsection{A norm identity}
We denote by $\|\cdot\|_{L^2}$ the norm in the space $L^2(\C)$. For a function
$F$, we let $\Mop_F$ denote the operator of multiplication by $F$. The first
step is the following norm identity for $v\in C^\infty_c(\C)$:
\begin{equation} 
\|\dbar v-v\dbar\phi\|^2_{L^2}-\|\partial v+v\partial\phi\|^2_{L^2}=
2\int_\C|v|^2\hDelta\phi \diff A.
\label{eq-normid1}
\end{equation}
To arrive at \eqref{eq-normid1}, we do as follows. As for the first term, 
if we let $\langle\cdot,\cdot\rangle_{L^2}$ denote the standard sesquilinear
inner product of $L^2(\C)$, we see that
\[
\|\dbar v-v\dbar\phi\|^2_{L^2}=\langle \dbar v-v\dbar\phi,\dbar v-v\dbar\phi
\rangle_{L^2}=\langle (\dbar-\Mop_{\dbar\phi})v,(\dbar-\Mop_{\dbar\phi})v
\rangle_{L^2}=\langle (\dbar-\Mop_{\dbar\phi})^*
(\dbar-\Mop_{\dbar\phi})v,v\rangle_{L^2}
\]
and together with the corresponding calculation for the second term, we
find that \eqref{eq-normid1} expresses that
\begin{equation}
(\dbar-\Mop_{\dbar\phi})^*(\dbar-\Mop_{\dbar\phi})-
(\partial+\Mop_{\partial\phi})^*
(\partial+\Mop_{\partial\phi})=2\Mop_{\hDelta\phi}.
\label{eq-normid2}
\end{equation}  
Here, the adjoints involved are readily expressed: $\partial^*=-\dbar$, 
$\dbar^*=-\partial$, and $\Mop_F^*=\Mop_{\bar F}$. The product rule says that
$\dbar\Mop_F=\Mop_{\dbar F}+\Mop_F\dbar$ and 
$\partial\Mop_F=\Mop_{\partial F}+\Mop_F\partial$. By identifying adjoints
and carrying out the necessary algebraic manipulations, \eqref{eq-normid2}
is immediate. 

\subsection{Reduction to a norm inequality}
While the norm identity \eqref{eq-normid1} is interesting by itself, 
we observe here that it has the consequence that
\begin{equation} 
2\int_\C|v|^2\hDelta\phi\diff A\le\|\dbar v-v\dbar\phi\|^2_{L^2},\qquad
v\in C^\infty_c(\C).
\label{eq-normest3}
\end{equation}
We write $\Tope:=\dbar-\Mop_{\dbar\phi}$, and express \eqref{eq-normest3} again:
\begin{equation} 
2\int_\C|v|^2\hDelta\phi\diff A\le\|\Tope v\|^2_{L^2},\qquad
v\in C^\infty_c(\C).
\label{eq-normest4}
\end{equation}
If $v_j$ is a sequence of functions such that $\Tope v_j$ converges in $L^2
(\C)$, then by \eqref{eq-normest4} the functions $v_j\sqrt{\hDelta\phi}$
converge in $L^2(\C)$ as well, and in particular, $v_j$ converges locally as
an $L^2$-function. It follows that if $h\in L^2(\C)$ is in to the 
$L^2(\C)$-closure of $\Tope C^\infty_c(\C)$, then there exists a function 
$v\in L^2(\C,\hDelta\phi)$ such that $\Tope v=h$ in the sense of distribution
theory, with
\begin{equation} 
2\int_\C|v|^2\hDelta\phi\diff A\le\|h\|^2_{L^2}.
\label{eq-normest5}
\end{equation}
It remains to identify the $L^2(\C)$-closure of $\Tope C^\infty_c(\C)$. 
To this end, we identify the orthogonal complement of $\Tope C^\infty_c(\C)$. 
So, let $k\in L^2(\C)$ be such that 
\begin{equation}
\langle k,\Tope v\rangle_{L^2}=0,\qquad v\in C^\infty_c(\C).
\label{eq-duality1}
\end{equation}
If we write $\Tope^*:=-\partial-\Mop_{\partial\phi}$, distribution theory 
gives that \eqref{eq-duality1} is the same as
\[
\langle \Tope^* k,v\rangle_{L^2}=0,\qquad v\in C^\infty_c(\C),
\]
which in its turn expresses that $\Tope^*k=0$ holds in the sense of 
distributions. Let us write $\mathrm{ker}\,\Tope^*$ for the space of all
$k\in L^2(\C)$ with $\Tope^*k=0$. We have arrived at the following result.

\begin{thm}
Suppose $h\in L^2(\C)\ominus\mathrm{ker}\,\Tope^*$. Then there exists a 
function $v\in L^2(\C,\hDelta\phi)$ such that $\Tope v=h$ with
\[
2\int_\C|v|^2\hDelta\phi\diff A\le\|h\|^2_{L^2}.
\]
\label{thm-H3}
\end{thm}

Since $\Tope^*=-\Mop_{\e^{-\phi}}\partial\Mop_{\e^{\phi}}$, 
$k\in\mathrm{ker}\,\Tope^*$ holds if and only if $\e^{\phi}\bar k\in 
A^2(\C,\e^{-2\phi})$.  

\begin{proof}[Proof of Theorem \ref{thm-H2}]
We put $h:=\e^\phi f\in L^2(\C)$. By condition \eqref{eq-dual1} for all 
$g\in A^2(\C,\e^{-2\phi})$, we have -- if we write $k=\e^{-\phi}\bar g$ -- 
that
\[
0=\int_\C fg\diff A=\int_\C h\bar k\diff A.
\]
As the function $g$ runs over $A^2(\C,\e^{-2\phi})$, $k$ runs over all 
elements 
of $\mathrm{ker}\,\Tope^*$. We conclude that $h\in 
L^2(\C)\ominus\mathrm{ker}\,\Tope^*$, so that Theorem \ref{thm-H3} applies,
and gives a $v\in L^2(\C,\Delta\phi)$ with $\Tope v=h$ with the norm control
\begin{equation}
2\int_\C|v|^2\hDelta\phi\diff A\le\|h\|^2_{L^2}.
\label{eq-dbarest}
\end{equation}
We put $u:=\e^{-\phi}v$. The operator $\Tope$ factors 
$\Tope=\Mop_{\e^\phi}\dbar
\Mop_{\e^{-\phi}}$, which means that the equation $\Tope v=h$ is equivalent to
$\dbar u=f$. Finally, the estimate \eqref{eq-dbarest} is equivalent to the
estimate
\[
2\int_\C|u|^2\e^{2\phi}\hDelta\phi\diff A\le\int_\C|f|^2\e^{2\phi}\diff A,
\]
which concludes the proof of the theorem. 
\end{proof}

\begin{proof}[Proof of Theorem \ref{thm-H2.1}]
We first observe that any two solutions of the $\bar\partial$-equation 
differ by an entire function. Moreover, under the given curvature-type 
condition, an entire function 
$F\in L^2(\C,\e^{2\phi}\Delta\phi)$ 
necessarily must vanish everywhere, by the following argument. The function 
$|F|^2\e^{2\phi}\Delta\phi$ is clearly nonnegative and it is also 
\emph{subharmonic in} $\C$. Indeed,
if $F$ is nontrivial, we have that [in the sense of distribution theory]
$\Delta\log|F|\ge0$, and consequently
\[
\Delta\log\big[|F|^2\e^{2\phi}\Delta\phi\big]=2\Delta\log|F|+2\Delta\phi
+\Delta\log\Delta\phi\ge0,
\]
which gives that the exponentiated function $|F|^2\e^{2\phi}\Delta\phi$ is 
subharmonic, as claimed. In the remaining case when $F$ vanishes identically
the claim is trivial. Next, by the estimate of each solution of the 
$\bar\partial$-problem supplied by Theorem \ref{thm-H2}, it is given that
the function 
$|F|^2\e^{2\phi}\Delta\phi$ is in $L^2(\C)$. If $\D(z_0,r)$ denotes the open 
disk of radius $r$ about $z_0$, the sub-mean vakue property of subharmonic
functions gives that
\[
|F(z_0)|^2\e^{2\phi(z_0)}\Delta\phi(z_0)\le \frac{1}{\pi r^2}\int_{\D(z_0,r)}
|F|^2\e^{2\phi}\Delta\phi \diff A\le 
\frac{1}{\pi r^2}\int_{\C}|F|^2\e^{2\phi}\Delta\phi\diff A.
\]
Letting $r\to+\infty$, we see that the left hand side vanishes. As $z_0$ is
arbitrary, it follows that $F(z)\equiv0$. This completes the proof.
\end{proof}

\begin{rem}
In Lemma 3.1 of \cite{AlRoz} appears a condition which is analogous to
\eqref{eq-dual1} in the context of distributions with compact support. 
Also, in the related paper \cite{RozShir}, the Fock weight case 
$\phi(z)=\frac12|z|^2$ is considered in the soft-topology setting of 
solutions which are distributions.  
\end{rem}

\begin{rem}
With $\phi=0$, the norm identity \eqref{eq-normid1} expresses that the
Beurling operator is an isometry on $L^2(\C)$, which implies the Grunsky 
inequalities in the theory of conformal mapping (see, e.g., \cite{BarHed};
cf. \cite{BerSch}). 
In the (somewhat singular) case when $\phi(z)=\theta\log|z|$, 
\eqref{eq-normid1} provides the main norm identity of \cite{Hed2}, which 
leads to a Prawitz-Grunsky type inequality for conformal maps. The most general 
inequality of this type (with multiple ``branch points'') for conformal maps 
was obtained by Shimorin \cite{Shim} (see also \cite{AbuHed}). 
It would be of interest to see if the results of Shimorin may be obtained 
from the general norm identity \eqref{eq-normid1}.   
\end{rem}

\section{Discussion of the necessity of the orthogonality 
condition on the datum}
\label{sec-constr}

\subsection{The Fock weight case}
We now narrow down the discussion to the Fock weight $\phi(z)=\frac12|z|^2$. 
Since the weight is radial, the polynomials are dense in 
$A^2(\C,\e^{-|z|^2})$. Also, the curvature-type condition of Theorem 
\ref{thm-H2.1} is readily checked.   
It follows that Theorems \ref{thm-H2} and \ref{thm-H2.1} combine to give
the following result.

\begin{thm}
Suppose $f\in L^2(\C,\e^{|z|^2})$. If the datum $f$ satisfies the moment 
condition
\[
\int_\C z^jf(z)\diff A(z)=0,\qquad j=0,1,2,\ldots,
\]
then there exists a unique solution $u$ to the equation $\dbar u=f$ 
with
\[
\int_\C|u(z)|^2\e^{|z|^2}\diff A(z)\le \int_\C|f(z)|^2\e^{|z|^2}
\diff A(z).
\]
\label{thm-BF0}
\end{thm}

We may ask what would happen if the orthogonality condition is not satisfied.
Maybe there still exists some rapidly decaying solution $u$ anyway? The answer
is definitely no.

\begin{thm}
Suppose $f\in L^2(\C,\e^{|z|^2})$, and that $u$ solves the equation $\dbar u
=f$ while $u\in L^2(\C,\e^{\epsilon|z|^2})$ for some positive real $\epsilon$. 
Then the datum $f$ has
\[
\int_\C z^jf(z)\diff A(z)=0,\qquad j=0,1,2,\ldots.
\]
\label{thm-BF1}
\end{thm}

Before we turn to the proof, we observe that if $f\in L^2(\R)$ with
\[
\int_\R |f(x)|\e^{x^2/\beta}\diff x<+\infty
\]
for some real $\beta>0$, then its Fourier transform 
\[
\hat f(\xi):=\int_\R \e^{-\imag\xi x}f(x)\diff x
\] 
extends to an entire function with
\[
\int_\C |\hat f(\xi)|^2\e^{-\beta|\im\xi|^2}\diff A(\xi)=
\frac{2\pi^{3/2}}{\sqrt{\beta}}\int_\R|f(x)|\e^{x^2/\beta}\diff x<+\infty.
\] 
In fact, the standard Bargmann transform theory asserts 
that this integrability condition characterizes the Fourier image of this 
weighted $L^2$ space on $\R$ (cf. \cite{Groech1}). The two-variable extension
of the above-mentioned result maintains that if $f\in L^2(\C,\e^{|z|^2/\beta})$,
then its Fourier transform
\[
\hat f(\xi,\eta):=\int_\C \e^{-\imag(\xi x+\eta y)}f(x+\imag y)\diff x\diff y
\]
is an entire function of two variables, with
\begin{equation}
\label{eq-BT2dim}
\iint_{\C\times\C} |\hat f(\xi,\eta)|^2\e^{-\beta(|\im\xi|^2+|\im\eta|^2)}
\diff A(\xi)\diff A(\eta)=
\frac{4\pi^{3}}{\beta}\int_\C|f(z)|\e^{|z|^2/\beta}\diff A(z)<+\infty.
\end{equation}

\begin{proof}[Proof of Theorem \ref{thm-BF1}]
Since $f\in L^2(\C,\e^{|z|^2})$, the function $\hat f$ is an entire function of
two variables with the norm bound \eqref{eq-BT2dim} with $\beta=1$. 
Likewise, as $u\in L^2(\C,\e^{\epsilon|z|^2})$ for some positive $\epsilon$,
$\hat u$ is an entire function of two variables. After Fourier transformation, 
the relation $\dbar u=f$ reads
\[
(\imag\xi -\eta)\hat u(\xi,\eta)=2\hat f(\xi,\eta).
\]
This is only possible if $\hat f(\xi,\eta)$ vanishes when $\imag\xi-\eta=0$,
i.e., $\hat f(\xi,\imag\xi)\equiv0$. By the definition of the Fourier 
transform, this means that
\[
0=\hat f(\xi,\imag\xi)=\int_\C \e^{-\imag\xi(x+\imag y)}f(x+\imag y)
\diff A(x+\imag y)=\int_\C \e^{-\imag\xi z}f(z)
\diff A(z)=\sum_{j=0}^{+\infty}\frac{(-\imag\xi)^j}{j!}\int_\C
z^jf(z)\diff A(z).
\]
By Taylor's formula, then, this implies that
\[
\int_\C z^jf(z)\diff A(z)=0,\qquad j=0,1,2,\ldots,
\]
as needed. 
\end{proof}

\begin{rem}
The constant in Theorem \ref{thm-BF0} is sharp. Indeed, we may consider the
datum $f(z)=-z\e^{-|z|^2}$ for which the solution is $u(z)=\e^{-|z|^2}$.
We calculate that
\[
\int_\C|u(z)|^2\e^{|z|^2}\diff A(z)=\int_\C\e^{-|z|^2}\diff A(z)=\pi
\]
while
\[
\int_\C|f(z)|^2\e^{|z|^2}\diff A(z)=\int_\C|z|^2\e^{-|z|^2}\diff A(z)=\pi,
\]
which gives the desired sharpness.
\end{rem}

\section{Acknowledgements}

The author wishes to thank Ioannis Parissis and Serguei Shimorin 
\cite{ParShi} for stimulating discussions related to the norm identity 
\eqref{eq-normid1}. 
The author also wishes to thank Grigori Rozenblum for an enlightening 
conversion on the topic of this paper at the Euler Institute in 
St-Petersburg in July, 2013.

\end{document}